\documentclass[12pt,reqno]{amsart}%
\usepackage{amsfonts,anysize}
\usepackage{mathptmx}
\usepackage[colorlinks,urlcolor=blue,citecolor=blue,linkcolor=blue]{hyperref}

\marginsize{2cm}{2cm}{1cm}{1cm}

\newtheorem{thm}{Theorem}[section]
\newtheorem{cor}[thm]{Corollary}

\theoremstyle{definition}

\theoremstyle{remark}

\numberwithin{equation}{section}
\allowdisplaybreaks
\begin{document}

\title{On the change of variables formula for multiple integrals}
\author{Shibo Liu
\and Yashan Zhang}\thanks{Email Addresses: {\tt liusb@xmu.edu.cn} (S.~Liu), {\tt colourful2009@163.com} (Y.~Zhang)}
\dedicatory{Department of Mathematics, Xiamen University, Xiamen 361005,
China\\
Department of Mathematics, University of Macau, Macau, China}

\begin{abstract}We develop an elementary proof of the change of variables formula in multiple integrals. Our proof is based on an induction argument. Assuming the formula for $(m-1)$-integrals, we define the integral over hypersurface in $\mathbb{R}^m$, establish the divergent theorem and then use the divergent theorem to prove the formula for $m$-integrals. In addition to its simplicity, an advantage of our approach is that it yields the Brouwer Fixed
Point Theorem as a corollary.
\end{abstract}
\maketitle

\section{Introduction}

The change of variables formula for multiple integrals is a fundamental
theorem in multivariable calculus. It can be stated as follows.

\begin{thm}
\label{t1}Let $D$ and $\Omega$ be bounded \emph{open} domains in
$\mathbb{R}^{m}$ with piece-wise $C^{1}$-boundaries, $\varphi\in C^{1}%
(\bar{\Omega},\mathbb{R}^{m})$ such that $\varphi:\Omega\rightarrow D$ is a
$C^{1}$-diffeomorphism. If $f\in C(\bar{D})$, then%
\begin{equation}
\int_{D}f(y)\mathrm{d}y=\int_{\Omega}f(\varphi(x))\left\vert J_{\varphi
}(x)\right\vert \mathrm{d}x\text{,} \label{e00}%
\end{equation}
where $J_{\varphi}(x)=\det\varphi^{\prime}(x)$ is the Jacobian determinant of
$\varphi$ at $x\in\Omega$.
\end{thm}

The usual proofs of this theorem that one finds in advanced calculus textbooks
involves careful estimates of volumes of images of small cubes under the map
$\varphi$ and numerous annoying details. Therefore several alternative proofs
have appeared in recent years. For example, in \cite{MR1699248} P. Lax proved
the following version of the formula.

\begin{thm}
Let $\varphi:\mathbb{R}^{m}\rightarrow\mathbb{R}^{m}$ be a $C^{1}$-map such
that $\varphi(x)=x$ for $\left\vert x\right\vert \geq R$, and $f\in
C_{0}(\mathbb{R}^{m})$. Then%
\[
\int_{\mathbb{R}^{m}}f(y)\mathrm{d}y=\int_{\mathbb{R}^{m}}f(\varphi
(x))J_{\varphi}(x)\mathrm{d}x\text{.}%
\]

\end{thm}

The requirment that $\varphi$ is an identity map outside a big ball is
somewhat restricted. This restriction was also removed by Lax in
\cite{MR1818184}. Then, Tayor \cite{MR1893211} and Ivanov\cite{MR2179859}
presented different proofs of the about result of Lax \cite{MR1699248} using
differential forms. See also B{\'a}ez-Duarte \cite{MR1231489} for a proof of a
variant of Theorem \ref{t1} which does not require that $\varphi:\Omega\to D$
is a diffeomorphism. As pointed out by Taylor \cite[Page 380]{MR1893211}, because the proof relies on integration of differential
forms over manifolds and Stokes' theorem, it requires that one knows the change of variables
formula as formulated in our Theorem \ref{t1}.

In this paper, we will present a simple elementary proof of Theorem \ref{t1}.
Our approach does not involve the language of differential forms. The idea is
motivated by Excerise 15 of \S 1-7 in the famous textbook on classical
differential geometry \cite{MR0394451} by do Carmo. The excerise deals with
the two dimensional case $m=2$. We will perform an induction argument to
generize the result to the higher dimensional case $m\geq2$. In our argument,
we will apply the Cauchy-Binet formula about the determinant of the product of
two matrics. As a byproduct of our approach, we will also obtain the
Non-Retraction Lemma (see Corollary \ref{c1}), which implies the Brouwer Fixed
Point Theorem.

\section{Integral over hypersurface}

We will prove Theorem \ref{t1} by an induction argument. The case $m=1$ is
easily proved in single variable calculus. Suppose we have proven Theorem
\ref{t1} for $\left(  m-1\right)  $-dimension, where $m\geq2$. We will define
the integral over a hypersurface (of codimension one) in $\mathbb{R}^{m}$ and establish the divergent theorem in $\mathbb{R}^{m}$. Then, in the next two sections we will use the divergent theorem to
prove Theorem \ref{t1} for $m$-dimension.

Let $U$ be a Jordan measurable bounded closed domain in $\mathbb{R}^{m-1}$, $x:U\rightarrow\mathbb{R}^{m}$,%
\[
(  u^{1},\ldots,u^{m-1})  \mapsto(  x^{1},\ldots,x^{m})
\]
be a $C^{1}$-map such that the restriction of $x$ in the interior $U^\circ$ is injective, and
\begin{equation}
\operatorname*{rank}\left(  \frac{\partial x^{i}}{\partial u^{j}}\right)
=m-1\text{,} \label{e}%
\end{equation}
then we say that $x:U\rightarrow\mathbb{R}^{m}$ is a $C^{1}$-parametrized
surface. By definition, two $C^{1}$-parametrized surfaces $x:U\rightarrow
\mathbb{R}^{m}$ and $\tilde{x}:\tilde{U}\rightarrow\mathbb{R}^{m}$ are
equivalent if there is a $C^{1}$-diffeomorphism $\phi:\tilde{U}\rightarrow U$
such that $\tilde{x}=x\circ\phi$. The equivalent class $[x]$ is called a
\emph{hypersurface}, and $x:U\rightarrow\mathbb{R}^{m}$ is called a
parametrization of the hypersurface. Since it is easy to see that
$x(U)=\tilde{x}(\tilde{U})$ if $x$ and $\tilde{x}$ are equivalent, $[x]$ can
be identified as the subset $S=x(U)$.

Let $S$ be a hypersurface with parametrization $x:U\rightarrow\mathbb{R}^{m}$.
By \eqref{e}, for $u\in U$,%
\begin{equation}
N(u)=\left(  \frac{\partial(x^{2},\ldots,x^{m})}{\partial(u^{1},\ldots
,u^{m-1})},\ldots,\left(  -1\right)  ^{m+1}\frac{\partial(x^{1},\ldots
,x^{m-1})}{\partial(u^{1},\ldots,u^{m-1})}\right)  \neq0\text{,} \label{e0}%
\end{equation}
where%
\[
\frac{\partial(x^{1},\ldots,\hat{x}^{i},\ldots,x^{m})}{\partial(u^{1}%
,\ldots,u^{m-1})}=\det\left(
\begin{array}
[c]{ccc}%
\partial_{u^{1}}x^{1} & \cdots & \partial_{u^{m-1}}x^{1}\\
\vdots &  & \vdots\\
\partial_{u^{1}}x^{i-1} & \cdots & \partial_{u^{m-1}}x^{i-1}\\
\partial_{u^{1}}x^{i+1} & \cdots & \partial_{u^{m-1}}x^{i+1}\\
\vdots &  & \vdots\\
\partial_{u^{1}}x^{m} & \cdots & \partial_{u^{m-1}}x^{m}%
\end{array}
\right)  \text{.}%
\]
It is well known that $N(u)$ is a normal vector of $S$ at $x(u)$.

Now, we can define the \emph{surface integral} of a continuous function
$f:S\rightarrow\mathbb{R}$ by%
\begin{equation}
\int_{S}f\,\mathrm{d}\sigma=\int_{U}f(x(u))\left\vert N(u)\right\vert
\,\mathrm{d}u\text{.} \label{ei}%
\end{equation}
By the change of variables formular for $\left(  m-1\right)  $-integrals, it
is not difficult to see that if $\tilde{x}:\tilde{U}\rightarrow\mathbb{R}^{m}$
is another parametrization of $S$, then%
\[
\int_{U}f(x(u))\left\vert N(u)\right\vert \,\mathrm{d}u=\int_{\tilde{U}%
}f(\tilde{x}(v))\left\vert \tilde{N}(v)\right\vert \,\mathrm{d}v\text{,}%
\]
where $\tilde{N}$ is defined similar to \eqref{e0}. Therefore, our surface
integral is well defined.

If $\Sigma=\bigcup_{i=1}^{\ell}S_{i}$, where $S_{i}=x_{i}(U_{i})$
are hypersurfaces such that $x_{i}(U_{i}^\circ)\cap x_{j}(U_{j}^\circ)=\emptyset$ for $i\neq j$,
then we call $\Sigma$ a piece-wise $C^{1}$-hypersurface and define the
integral of $f\in C(\Sigma)$ by%
\[
\int_{\Sigma}f\,\mathrm{d}\sigma=\sum_{i=1}^{\ell}\int_{S_{i}}f\,\mathrm{d}%
\sigma\text{.}%
\]

\begin{thm}
[Divergent Theorem]\label{t2}Let $D$ be bounded open domain in $\mathbb{R}%
^{m}$ with piece-wise $C^{1}$-boundary $\partial D$, $F:\bar{D}\rightarrow
\mathbb{R}^{m}$ be a $C^{1}$-vector field, $n$ is the unit outer normal vector
field on $\partial D$, then%
\[
\int_{D}\operatorname*{div}F\,\mathrm{d}x=\int_{\partial D}F\cdot
n\,\mathrm{d}\sigma\text{.}%
\]
\end{thm}

\begin{proof}
Having defined the surface integral, the proof of the theorem is a standard
application of the Fubini Theorem. We include the details here for completeness.

We say that $F=(F^{1},\ldots,F^{m})$ is of $i$-type if $F^{j}=0$ for $j\neq
i$. We also say that $D$ is of $i$-type, if there are a bounded closed domain
$U$ in $\mathbb{R}^{m-1}$ with piece-wise $C^{1}$-boundary and $\varphi_{\pm
}\in C^{1}(U)$ such that%
\[
D=\left\{  \left.  x\right\vert \,\varphi_{-}(x^{\prime})< x^{i}<
\varphi_{+}(x^{\prime})\text{, }x^{\prime}\in U^\circ\right\}  \text{,}%
\]
where $x^{\prime}=(x^{1},\ldots,x^{i-1},x^{i+1},\ldots,x^{m})$.

Let $F=(0,\ldots,0,F^{m})$ be an $m$-type vector field. Suppose $D$ is of
$m$-type with $U$ and $\varphi_{\pm}$ as above. Then $\partial D$ consists of
three parts:%
\[
\Sigma_{\pm}=\left\{  \left.  x=(x^{\prime},\varphi_{\pm}(x^{\prime
}))\right\vert \,x^{\prime}\in U\right\}
\]
and%
\[
\Sigma_{0}=\left\{  \left.  x=(x^{\prime},x^{m})\right\vert \,\varphi
_{-}(x^{\prime})\leq x^{m}\leq\varphi_{+}(x^{\prime})\text{, }x^{\prime}%
\in\partial U\right\}  \text{.}%
\]
On $\Sigma_{\pm}$, by \eqref{e0} we have%
\[
N=\left(
-1\right)  ^{m+1}\left(  -\partial_{1}\varphi_{\pm},\ldots,-\partial_{m-1}\varphi_{\pm},1\right)  \text{.}%
\]
Hence $\left\vert N\right\vert =\sqrt{1+\left\vert \nabla\varphi_{\pm
}\right\vert ^{2}}$ and%
\[
n=\pm\frac{1}{\sqrt{1+\left\vert \nabla\varphi_{\pm}\right\vert ^{2}}}\left(
-\partial_{1}\varphi_{\pm},\ldots,-\partial_{m-1}\varphi_{\pm},1\right)  \text{.}%
\]
While on $\Sigma_{0}$, $n=(--,0)$ and $F\cdot n=0$. Consequently, by
\eqref{ei} we obtain%
\begin{align*}
\int_{\partial D}F\cdot n\,\mathrm{d}\sigma &  =\int_{\Sigma_{+}}F\cdot
n\,\mathrm{d}\sigma+\int_{\Sigma_{-}}F\cdot n\,\mathrm{d}\sigma+\int
_{\Sigma_{0}}F\cdot n\,\mathrm{d}\sigma\\
&  =\int_{U}F^{m}(x^{\prime},\varphi_{+}(x^{\prime}))\mathrm{d}x^{\prime}%
-\int_{U}F^{m}(x^{\prime},\varphi_{-}(x^{\prime}))\mathrm{d}x^{\prime}\\
&  =\int_{U}\mathrm{d}x^{\prime}\int_{\varphi_{-}(x^{\prime})}^{\varphi
_{+}(x^{\prime})}\partial_{m}F^{m}(x^{\prime},t)\mathrm{d}t =\int_{D}%
\partial_{m}F^{m}(x)\mathrm{d}x=\int_{D}\operatorname*{div}F\,\mathrm{d}%
x\text{.}%
\end{align*}
In a similar maner we can show that the theorem is valid for $i$-type vector
field on $i$-type domain.

As in most calculus textbooks, we only prove the theorem for the case that $D$
is simultaneously $i$-type for all $i=1,\ldots,m$. For a general $C^{1}%
$-vector field $F=\left(  F^{1},\ldots,F^{m}\right)  $ on $\bar{D}$, we set
$F_{i}=(0,\ldots,F^{i},\ldots,0)$. Since $F=F_{1}+\cdots+F_{m}$, and $F_{i}$
is $i$-type vector field on $i$-type domain $D$, we deduce
\begin{align*}
  \int_{\partial D}F\cdot n\,\mathrm{d}\sigma&=\sum_{i=1}^m\int_{\partial D} F_i\cdot n\,\mathrm{d}\sigma=\sum_{i=1}^m
  \int_D\operatorname*{div}F_i\,\mathrm{d}x
  =\int_D\operatorname*{div}F\,\mathrm{d}x\text{.}\qedhere
\end{align*}

\end{proof}

\section{Domains with singly parametrized boundary}

In this section, we prove the $m$-dimensional change of variables formula
\eqref{e00} for the case that $\partial\Omega$ can be \emph{singly
parametrized}, that is, there exists a $C^1$-parametrized surface $x:U\rightarrow\mathbb{R}^{m}$ such
that $\partial\Omega=x(U)$. For example, if $\Omega$ is a ball,
then $\partial\Omega$ can be singly parametrized by the well known parametrization.

In this case, we only need to require that the transformation $\varphi$ maps
$\partial\Omega$ to $\partial D$ diffeomorphicly. We have the following theorem.

\begin{thm}
\label{t3}Let $D$ and $\Omega$ be bounded open domains in $\mathbb{R}^{m}$
with $C^{1}$-boundaries, $\partial\Omega$ can be singly parametrized. Suppose
$\varphi:\bar{\Omega}\rightarrow\bar{D}$ is a $C^{1}$-map so that $\varphi$
maps $\partial\Omega$ to $\partial D$ diffeomorphicly, and $f\in C(\bar{D})$,
then%
\begin{equation}
\int_{D}f(y)\mathrm{d}y=\pm\int_{\Omega}f(\varphi(x))J_{\varphi}%
(x)\,\mathrm{d}x\text{.}\label{e3}%
\end{equation}
Here, the choice of the signs $\pm$ on the right hand side depends on whether
$\varphi$ preserve the orientation of the boundaries.
\end{thm}

\begin{proof}
Since $f\in C(\bar{D})$, it can be continuously extended to $\mathbb{R}^{m}$.
Doing convolution with the mollifiers $\left\{  \eta_{\varepsilon}\right\}
_{\varepsilon>0}$, which are functions $\eta_\varepsilon\in C^\infty(\mathbb{R}^m)$ such that
\[
\int_{\mathbb{R}^{m}}\eta_{\varepsilon}(y)\mathrm{d}y=1\text{,\qquad
}\operatorname*{supp}\eta_{\varepsilon}\subset B_{\varepsilon}(0)\text{,}%
\]
we obtain a family of functions $f_{\varepsilon}\in C^{\infty}(\mathbb{R}%
^{m})$ such that as $\varepsilon\rightarrow0^{+}$,%
\[
\sup_{y\in D}\left\vert f_{\varepsilon}(y)-f(y)\right\vert \rightarrow
0\text{,\qquad}\sup_{x\in\Omega}\left\vert f_{\varepsilon}(\varphi
(x))J_{\varphi}(x)-f(\varphi(x))J_{\varphi}(x)\right\vert \rightarrow0\text{.}%
\]
It is then easy to see that%
\[
\int_{D}f_{\varepsilon}(y)\mathrm{d}y\rightarrow\int_{D}f(y)\mathrm{d}%
y\text{,\qquad}\int_{\Omega}f_{\varepsilon}(\varphi(x))J_{\varphi
}(x)\mathrm{d}x\rightarrow\int_{\Omega}f(\varphi(x))J_{\varphi}(x)\mathrm{d}%
x\text{.}%
\]
Therefore, we only need to prove \eqref{e3} for $f\in C^{\infty}%
(\mathbb{R}^{m})$. Using a similar approximating argument we may also assume that
$\varphi\in C^{2}(\bar{\Omega},\mathbb{R}^{m})$.

Let $C=(-a,a)\times\cdots\times(-a,a)$ be a cube containing $\bar{D}$, then define $Q:\bar{D}\rightarrow\mathbb{R}$,
\[
Q(y)=\int_{-a}^{y^1}f(t,y^2,\ldots,y^m)\mathrm{d}t\text{.}
\]
Then $Q\in
C^{1}(\bar{D})$ and
\[
\frac{\partial Q}{\partial y^{1}}=f\text{\qquad in }D\text{.}%
\]

Let $x:U\rightarrow\mathbb{R}^{m}$ be a parametrization of $\partial\Omega$.
Since $\varphi$ maps $\partial\Omega$ to $\partial D$ diffeomorphicly, it
follows that $y=\varphi\circ x$ is a parametrization of $\partial D$. Then%
\[
N=\left(  \frac{\partial(y^{2},\ldots,y^{m})}{\partial(u^{1},\ldots,u^{m-1}%
)},\ldots,\left(  -1\right)  ^{m+1}\frac{\partial(y^{1},\ldots,y^{m-1}%
)}{\partial(u^{1},\ldots,u^{m-1})}\right)
\]
is a normal vector at $y(u)$ on $\partial D$ and%
\begin{equation}
n=\pm N/\left\vert N\right\vert =(  n^{1},\ldots,n^{m})  \label{e4}%
\end{equation}
is the unit outer normal vector at $y(u)$ on $\partial D$. By the chain role
we have%
\[
\left(
\begin{matrix}
\partial_{u^{1}}y^{2} & \cdots & \partial_{u^{m-1}}y^{2}\\
\vdots &  & \vdots\\
\partial_{u^{1}}y^{m} & \cdots & \partial_{u^{m-1}}y^{m}%
\end{matrix}
\right)  =\left(
\begin{matrix}
\partial_{x^{1}}y^{2} & \cdots & \partial_{x^{m}}y^{2}\\
\vdots &  & \vdots\\
\partial_{x^{1}}y^{m} & \cdots & \partial_{x^{m}}y^{m}%
\end{matrix}
\right)  \left(
\begin{matrix}
\partial_{u^{1}}x^{1} & \cdots & \partial_{u^{m-1}}x^{1}\\
\vdots &  & \vdots\\
\partial_{u^{1}}x^{m} & \cdots & \partial_{u^{m-1}}x^{m}%
\end{matrix}
\right)  \text{.}%
\]
Applying the Cauchy-Binet formular, we obtain from \eqref{e4} that%
\begin{align}
\pm n^{1}\left\vert N\right\vert  &  =\frac{\partial(y^{2},\ldots,y^{m}%
)}{\partial(u^{1},\ldots,u^{m-1})}\nonumber\\
&  =\sum_{i=1}^{m}\frac{\partial(y^{2},\ldots,y^{m})}{\partial(x^{1}%
,\ldots,\hat{x}^{i},\ldots,x^{m})}\frac{\partial(x^{1},\ldots,\hat{x}%
^{i},\ldots,x^{m})}{\partial(u^{1},\ldots,u^{m-1})}=A\cdot\tilde{N}%
\text{,}\label{e6}%
\end{align}
where $A=\left(  A_{1},\ldots A_{m}\right)  $, $\tilde{N}=\left(  \tilde
{N}^{1},\ldots,\tilde{N}^{m}\right)  $, with
\[
A_{i}=\left(  -1\right)  ^{i+1}\frac{\partial(y^{2},\ldots
,y^{m})}{\partial(x^{1},\ldots,\hat{x}^{i},\ldots,x^{m})}\text{,\qquad}%
\tilde{N}^{i}=\left(  -1\right)  ^{i+1}\frac{\partial(x^{1},\ldots,\hat{x}%
^{i},\ldots,x^{m})}{\partial(u^{1},\ldots,u^{m-1})}\text{.}%
\]
Note that $\tilde{n}=\pm\tilde{N}/\left\vert \tilde{N}\right\vert $ is the
unit outer normal vector at $x(u)$ on $\partial\Omega$. Moreover, $A_{i}$ is
exactly the algebraic cofactor of $\partial_{x^{i}}y^{1}$ in the Jacobian%
\[
J_{\varphi}(x)=\det\left(
\begin{array}
[c]{ccc}%
\partial_{x^{1}}y^{1} & \cdots & \partial_{x^{m}}y^{1}\\
\vdots &  & \vdots\\
\partial_{x^{1}}y^{m} & \cdots & \partial_{x^{m}}y^{m}%
\end{array}
\right)  \text{.}%
\]
Thus, since $\varphi$ is of class $C^{2}$, by the Hadamard identity \cite[Page
14]{MR787404} we deduce%
\begin{equation}
\operatorname*{div}A=\sum_{i=1}^{m}\frac{\partial A_{i}}{\partial x^{i}%
}=0\text{.}\label{e7}%
\end{equation}
Let $\tilde{Q}=Q\circ\varphi$, then $\tilde{Q}\in C^{1}(\bar{\Omega})$. Using
\eqref{e7} we obtain%
\begin{align*}
\operatorname*{div}(\tilde{Q}A) &  =\nabla\tilde{Q}\cdot A+\tilde
{Q}\operatorname*{div}A=\nabla\tilde{Q}\cdot A\\
&  =\sum_{i=1}^{m}\frac{\partial\tilde{Q}}{\partial x^{i}}A_{i}=\sum_{i=1}%
^{m}\left(  \sum_{j=1}^{m}\left.  \frac{\partial Q}{\partial y^{j}}\right\vert
_{\varphi(x)}\frac{\partial y^{j}}{\partial x^{i}}\right)  A_{i}\\
&  =\sum_{j=1}^{m}\left.  \frac{\partial Q}{\partial y^{j}}\right\vert
_{\varphi(x)}\left(  \sum_{i=1}^{m}\frac{\partial y^{j}}{\partial x^{i}}%
A_{i}\right)  =\sum_{j=1}^{m}\left.  \frac{\partial Q}{\partial y^{j}%
}\right\vert _{\varphi(x)}\delta_{1}^{j}J_{\varphi}(x)\\
&  =\left.  \frac{\partial Q}{\partial y^{1}}\right\vert _{\varphi
(x)}J_{\varphi}(x)=f(\varphi(x))J_{\varphi}(x)\text{.}%
\end{align*}
Applying Theorem \ref{t2} and using \eqref{e6}, we have%
\begin{align*}
\int_{D}f(y)\mathrm{d}y &  =\int_{D}\frac{\partial Q}{\partial y^{1}%
}\mathrm{d}y=\int_{\partial D}Qn^{1}\,\mathrm{d}\sigma\\
&  =\int_{U}Q(y(u))n^{1}(u)\left\vert N(u)\right\vert \mathrm{d}u\\
&  =\pm\int_{U}\tilde{Q}(x(u))\left(  A(x(u))\cdot\tilde{N}(u)\right)
\,\mathrm{d}u\\
&  =\pm\int_{U}\left(  \tilde{Q}(x(u))A(x(u))\cdot\tilde{n}(u)\right)
\left\vert \tilde{N}(u)\right\vert \,\mathrm{d}u\\
&  =\pm\int_{\partial\Omega}\tilde{Q}A\cdot\tilde{n}\,\mathrm{d}\sigma=\pm
\int_{\Omega}\operatorname*{div}(\tilde{Q}A)\,\mathrm{d}x=\pm\int_{\Omega
}f(\varphi(x))J_{\varphi}(x)\,\mathrm{d}x\text{.}\qedhere
\end{align*}

\end{proof}

\begin{cor}
\label{cc}Under the assumption of Theorem \ref{t3}, if $J_{\varphi}(x)$ does
not change sign as $x$ varies in $\Omega$, then%
\[
\int_{D}f(y)\mathrm{d}y=\int_{\Omega}f(\varphi(x))\left\vert J_{\varphi
}(x)\right\vert \,\mathrm{d}x\text{.}%
\]

\end{cor}

\begin{cor}
[Non-Retraction Lemma]\label{c1}Let $B$ be the unit closed ball in
$\mathbb{R}^{m}$, then there does not exist $C^{1}$-map $T:B\rightarrow
\mathbb{R}^{m}$ such that $T(B)\subset\partial B$ and $T|_{\partial
B}=1_{\partial B}$.
\end{cor}

\begin{proof}
The proof below is essentially a variant form of the argument in
\cite[Corollarys]{MR1231489}. Suppose there is a $C^{1}$-map $T$ with the
stated properties. Obviously $T$ map $\partial B$ to itself diffeomorphicly.
We define a continuous function $f:B\rightarrow\mathbb{R}$,%
\[
f(y)=\left\{
\begin{array}
[c]{ll}%
1-4\left\vert y\right\vert ^{2}\text{,} & \text{if }\left\vert y\right\vert
\leq\frac{1}{2}\text{,}\\
0\text{,} & \text{if }\frac{1}{2}<\left\vert y\right\vert \leq1\text{.}%
\end{array}
\right.
\]
Then $f(T(x))=0$ for all $x\in B$. By Theorem \ref{t3},%
\[
0<\int_{B}f(y)\,\mathrm{d}y=\pm\int_{B}f(T(x))J_{T}(x)\,\mathrm{d}x=0\text{,}%
\]
a contradiction.
\end{proof}

As is well known, the Brouwer Fixed Point Theorem is an easy consequence of
Corollary \ref{c1}.

\section{General domains}

In this section, we prove Theorem \ref{t1} for the general case that
$\partial\Omega$ may not be singly parametrized. Let $f_{\pm}=\max\left\{  \pm
f,0\right\}  $, then $f=f_{+}-f_{-}$. Because $f_{\pm}$ are also continuous on
$\bar{D}$, it follows that we only need to prove the result for nonnegative
$f$. For simplicity, we set%
\[
\tilde{f}(x)=f(\varphi(x))\left\vert J_{\varphi}(x)\right\vert \text{.}%
\]
We want to prove%
\[
\int_{D}f(y)\,\mathrm{d}y=\int_{\Omega}\tilde{f}(x)\,\mathrm{d}x\text{.}%
\]
For any $\varepsilon>0$, there exist disjoint balls $B_{i}\subset\Omega$
($i=1,\ldots,\ell$) such that%
\begin{equation}
\int_{\Omega}\tilde{f}(x)\,\mathrm{d}x\leq\sum_{i=1}^{\ell}\int_{B_{i}}%
\tilde{f}(x)\,\mathrm{d}x+\varepsilon\text{.} \label{e8}%
\end{equation}
Let $U_{i}=\varphi(B_{i})$, then $\varphi:B_{i}\rightarrow U_{i}$ is a $C^{1}%
$-diffeomorphism. Since $\partial B_{i}$ can be singly parametrized and
$J_{\varphi}$ is of constant sign in $\Omega$, hence in $B_{i}$, by Corollary
\ref{cc} we have%
\begin{equation}
\int_{B_{i}}\tilde{f}(x)\,\mathrm{d}x=\int_{U_{i}}f(y)\,\mathrm{d}y\text{.}
\label{e9}%
\end{equation}
Because $U_{i}\cap U_{j}=\emptyset$ and%
\[
U=\bigcup_{i=1}^{\ell}U_{i}\subset D\text{,}%
\]
from \eqref{e8}, \eqref{e9}, and noting that $f\geq0$, we deduce that%
\[
\int_{\Omega}\tilde{f}(x)\,\mathrm{d}x\leq\sum_{i=1}^{\ell}\int_{U_{i}%
}f(y)\,\mathrm{d}y+\varepsilon=\int_{U}f(y)\,\mathrm{d}y+\varepsilon\leq
\int_{D}f(y)\,\mathrm{d}y+\varepsilon\text{.}%
\]
Let $\varepsilon\rightarrow0$, we get%
\begin{equation}
\int_{\Omega}\tilde{f}(x)\,\mathrm{d}x\leq\int_{D}f(y)\,\mathrm{d}y\text{.}
\label{e10}%
\end{equation}
Since $\varphi:\Omega\rightarrow D$ is a diffeomorphism, switching the roles
of $f$ and $\tilde{f}$ in the above argument, we obtain%
\begin{equation}
\int_{D}f(y)\,\mathrm{d}y\leq\int_{\Omega}\tilde{f}(x)\,\mathrm{d}x\text{.}
\label{e11}%
\end{equation}
Now the conclusion of Theorem \ref{t1} follows from \eqref{e10} and \eqref{e11}.

\subsection*{Acknowledgments}
This work was supported by NSFC (11671331) and NSFFJ (2014J06002). Y. Zhang is grateful to Professor Huai-Dong Cao for constant encouragement and support. He is supported by the Science and Technology Development Fund (Macao S.A.R.) Grant FDCT/ 016/2013/A1 and the Project MYRG2015-00235-FST of the University of Macau.

\end{document}